\documentclass[12pt,oneside]{amsart}
\usepackage{hyperref}

\usepackage{cmupint}

\usepackage[a4paper, margin=1in]{geometry}

\usepackage{comment}

\newtheorem{theorem}{Theorem}[section]
\newtheorem{proposition}[theorem]{Proposition}
\newtheorem{lemma}[theorem]{Lemma}
\newtheorem{corollary}[theorem]{Corollary}

\theoremstyle{definition}

\newtheorem{remark}[theorem]{Remark}

\newtheorem{question}[theorem]{Question}

\newtheorem{innercustomass}{Assumption}
\newenvironment{assumption-w-label}[1]
  {\renewcommand\theinnercustomass{#1}\innercustomass}
  {\endinnercustomass}

\newcommand{\ga}{\alpha}

\DeclareMathOperator{\diam}{diam}
\DeclareMathOperator{\Mob}{M\mathnormal{\ddot{\mathrm{o}}}b}
\def\D#1{|#1'|}
\def\Db#1{|(#1)'|}
\def\eps{\epsilon}
\DeclareMathOperator{\re}{Re}
\DeclareMathOperator{\ess}{ess}
\newcommand{\fratop}[2]{\genfrac{}{}{0pt}{}{#1}{#2}}

\title{Boundary representations of hyperbolic groups: the log--Sobolev case}
\author{Kevin Boucher}
\address{School of Mathematical Sciences, University of Southampton, Highfield, Southampton, SO17 1BJ, United Kingdom}
\email{kevin.boucher01@gmail.com}
\author{J\'{a}n \v{S}pakula}
\address{School of Mathematical Sciences, University of Southampton, Highfield, Southampton, SO17 1BJ, United Kingdom}
\email{jan.spakula@soton.ac.uk}
\date{\today}

\begin{document}

\begin{abstract}
    We study boundary representations of hyperbolic groups $\Gamma$ on the (compactly embedded) function space $W^{\log,2}(\partial\Gamma)\subset L^2(\partial\Gamma)$, the domain of the logarithmic Laplacian on $\partial\Gamma$. We show that they are not uniformly bounded, and establish their exact growth (up a multiplicative constant):  they grow with the square root of the length of $g\in\Gamma$. We also obtain $L^p$--analogue of this result.
    Our main tool is a logarithmic Sobolev inequality on bounded Ahlfors--David regular metric measure spaces.
\end{abstract}

\maketitle

\section{Introduction}\label{sec:intro}

To state the main results of this paper, let us start with a Gromov hyperbolic group $\Gamma$, acting properly, cocompactly, and by isometries, on a Gromov hyperbolic space $(X,\rho)$. The Gromov boundary $\partial X$ can be endowed with a visual metric $d$, and a Radon measure $\nu$, such that action of $\Gamma$ on $\partial X$ preserves the class of $\nu$ (i.e. maps sets of measure zero to sets of measure zero), hence the Radon--Nikodym derivative of the action of any $g\in \Gamma$ exists. Thus, for any $z\in\mathbb{C}$, $0\leq \re(z)\leq 1$, the formula
\begin{equation}\label{eq:pi_z}
	\pi_z(g)\phi(\xi) = \left(\tfrac{\mathrm{d}g_*\nu}{\mathrm{d}\nu}\right)^z\!(\xi) \cdot \phi(g^{-1}\xi)
\end{equation}
where $\xi\in \partial X$, $\phi:\partial X\to\mathbb{C}$, $g\in \Gamma$, formally defines a representation of $\Gamma$ on functions on $\partial X$. These are the so-called \emph{boundary representations}.

\medskip

In this paper, we focus on the case $\re(z)=\frac{1}{2}$ --- these are of course isometric on $L^2$; in the context of Lie groups they would be referred to as "principal series representations". The representation $\pi_{\frac{1}{2}}$ is sometimes referred to as the Koopman representation, or the quasi--regular representation. Nevertheless, there is a (compactly embedded) dense subspace of $L^2(\partial X)$ which is preserved under the representation, namely the domain of the so-called logarithmic Laplacian \cite{Gerontogiannis-Mesland:2023}; we denote this space $W^{\log,2}(\partial X)$ in this paper. We are interested in norm bounds for $\pi_z$ for the natural norm on $W^{\log,2}$ (which makes it into a Hilbert space).
We also address the $L^p$--analogue of this setup, $p\geq 1$.

\medskip

Let us provide some context for considering the space $W^{\log,2}$. On $\mathbb{R}^n$, there are fractional Sobolev spaces $W^{s,2}$ for $s\in(0,1)$ which interpolate between $L^2$ and the standard $W^{1,2}$. One way to describe the norm on $W^{s,2}$ is by way of a non-local Dirichlet form given by the convolution kernel $|\cdot|^{-n-2s}$. See for example \cite[Section 2]{Sobolev-Hitchhikers}. When appropriately normalised, the norms indeed converge to the $L^2$ and $W^{1,2}$ norms for $s\to0^{+}$ and $s\to 1^{-}$ respectively (see e.g.~\cite[p.9]{Sobolev-Hitchhikers} and references therein). However when the norms are \emph{not} normalised, for $s\to0^+$ they converge to the norm coming from the kernel $|\cdot-\cdot|^{-n}$, leading to the function space that we refer to $W^{\log,2}$ in this paper, cf.~\eqref{eq:defn-E_p,log} below.

In \cite{Boucher-Spakula:2023} the authors considered boundary representations on the fractional Sobolev spaces $W^{s,p}(\partial X)$ and proved that they are uniformly bounded. This piece deals with the `limit case' of $s=0$.

\medskip

For our results, we need the setting to be a little nicer than an arbitrary Gromov hyperbolic space $(X,\rho)$.
Any Gromov hyperbolic group $\Gamma$ acts properly and cocompactly on a \emph{strongly hyperbolic space} $(X,\rho)$ (with some parameter $\eps>0$), cf.~\cite{Nica-Spakula:2016} and references therein. 
One of the main features of strong hyperbolicity that we need is its $\text{CAT}(-1)$--like behaviour at infinity: if $o,x,y\in X$, denote the Gromov product by $\langle x,y\rangle_o = \frac{1}{2}(\rho(o,x)+\rho(o,y)-\rho(x,y))$. Then $\langle \cdot,\cdot\rangle_o$ extends continuously to $\partial X$ and the formula $d(\xi,\eta)=\exp(-\eps\langle\xi,\eta\rangle_o)$, $\xi,\eta\in \partial X$, defines a metric on $\partial X$. This is not generally true when $X$ is not-strongly hyperbolic.

Finally, let $D$ be the Hausdorff dimension of $(\partial X,d)$, and $\nu$ be the $D$-dimensional Hausdorff measure on $\partial X$. Then $\nu$ is \emph{Ahlfors--David $D$-regular}, i.e. $\nu(B(\xi,r))\asymp r^D$ for $\xi\in\partial X$, $r\in[0,\diam(\partial X)]$ and $B(\xi,r)$ denoting the closed ball about $\xi$ with radius $r$. This is true whenever $X$ is hyperbolic \cite{Coornaert:1993}. 

We record this setup as an Assumption for easy reference.

\begin{assumption-w-label}{(Hyp)}\label{ass:3}
	Let $\Gamma$ be a non--elementary Gromov hyperbolic group, acting properly and cocompactly on a roughly geodesic, strongly hyperbolic space $(X,\rho)$ (with parameter $\eps$). Fix $o\in X$. Denote $Z=\partial X$; $d(\xi,\eta)=\exp(-\eps\langle\xi,\eta\rangle_o)$ for $\xi,\eta\in Z$; and $\nu$ the $D$-dimensional Hausdorff measure on $(Z,d)$, which is Ahlfors--David $D$-regular.
\end{assumption-w-label}

Note that under this Assumption, $\diam(Z)=\nu(Z)=1$.

Before stating the results, we need to introduce a different notation ("reparametrise") the boundary representations \eqref{eq:pi_z}: this better matches the established notation for function spaces (e.g. $L^p$)\footnote{When $p=2$, the "$z$" from \eqref{eq:pi_z} and "$s+it$" here are related by $s+it=D(\frac{1}{2}-z)$. In particular, $s=0$ corresponds to $\re(z)=\frac{1}{2}$.}.
Let $p\geq1$, $s\in[-D/p,D/p]$, $t\in\mathbb{R}$. Denote
\begin{equation}\label{eq:pi-p-formula-first-version}
	\pi^{(p)}_{s+it}(g)\psi(\xi)=\Db{g^{-1}}^{\frac{D}{p}-s-it}(\xi)\cdot \psi(g^{-1}\xi),
\end{equation}
where $g\in \Gamma$, $\psi:Z\to \mathbb{C}$, $\xi\in Z$, and $\Db{g^{-1}}: Z\to(0,\infty)$ is the \emph{metric derivative} of the action of $g$ on $Z$: see Subsection \ref{subsec:setup} below, or \cite[Section 4]{Nica:2013}. Under Assumption \ref{ass:3}, its power is exactly the Radon--Nikodym derivative of the action.

\medskip

The final ingredient for the results is notation for some function spaces on $Z$. For a function $\phi:Z\to\mathbb{C}$, we use the notation $\|\phi|\mathbf{F}\|$ to denote the norm of $\phi$ in the function space $\mathbf{F}$, for example $\|\phi|L^p\|$. For $1\leq p<\infty$ and $\phi:Z\to\mathbb{C}$, denote
\begin{equation}\label{eq:defn-E_p,log}
	\mathcal{E}_{\log,p}(\phi) = \iint_{Z^2}\frac{\left|\phi(\xi)-\phi(\eta)\right|^p}{d^D(\xi,\eta)}\,\mathrm{d}\nu(\xi)\mathrm{d}\nu(\eta).
\end{equation}
For $p=2$, this is a Dirichlet form \cite{Fukushima-Oshima-Takeda:2011}, whose domain contains all $\alpha$-H\"{o}lder continuous functions for $0<\alpha\leq 1$. We denote its domain in $L^p(Z)$ by $W^{\log,p}(Z)$. Endowed with the norm $\|\phi|W^{\log,p}\|^p = \|\phi|L^p\|^p+\mathcal{E}_{\log,p}(\phi)$, it is a Banach space, dense in $L^p$ (see \cite[Section 4]{Boucher-Spakula:2023} for arguments), and reflexive for $p>1$. The "log" refers to the fact that it is the domain of the so-called logarithmic Laplacian on $Z$, see \cite{Gerontogiannis-Mesland:2023}.

\begin{theorem}\label{thm:main-s=0}
	In the setting of Assumption \ref{ass:3}, let $1\leq p<\infty$ and $t\in\mathbb{R}$. Then we have $\|\pi^{(p)}_{it}(g)\mid W^{\log,p}\to W^{\log,p}\|^p \asymp_{t} 1+\rho(go,o)$ for any $g\in\Gamma$.
\end{theorem}

This is a precise estimate --- however one can show that these representations can not be uniformly bounded by a more general argument, as follows.
We observe in Proposition \ref{prop:main-subrep-of-ub} below that under mild hypotheses, a compactly embedded sub-rep\-re\-sen\-ta\-tion of a uniformly bounded representation can not be uniformly bounded. Let us recall some terminology first:
We say that a representation $\pi:\Gamma\to\mathcal{B}(V)$ of a group $\Gamma$ on a Banach space $V$ is \emph{uniformly bounded}, if $\sup_{g\in G}\|\pi(g)\| < \infty$. We say that $\pi$ is \emph{weakly mixing}, if for any $\theta\in V^{*}$ (the topological dual of $V$), any $v\in V$, and any $\varepsilon>0$, there exists $g\in\Gamma$ such that  $|\theta(\pi(g)v)|<\varepsilon$. We say that $\pi$ is \emph{mixing}, if $|\Gamma|=\infty$ and for any $\theta\in V^{*}$ and any $v\in V$, we have $\lim_{g\to\infty} \theta(\pi(g)v)=0$.

\begin{proposition}\label{prop:main-subrep-of-ub}
    Let $\pi: G\to \mathcal{B}(V)$ be a weakly mixing, uniformly bounded representation of a group $G$ on a Banach space $V$ with a separable dual $V^*$.
    Suppose that $W\subset V$ is preserved under $\pi$, and is endowed with a norm such that the inclusion $W\hookrightarrow V$ is compact.
    Then $\pi$ is not uniformly bounded on $W$.
\end{proposition}

\begin{remark}\label{rem:subrep-in-mixing-ub}
	By a similar argument as our proof of the above Proposition, one can prove the following variant: assume that $\pi$ is mixing and uniformly bounded, $V$ is an arbitrary Banach space, and $W\hookrightarrow V$ is compact.
    Then $\lim_{g\to\infty}\|\pi(g)\mid W\to W\|=\infty$.
\end{remark}

The above Proposition applies to our setup for $1<p<\infty$, since $\pi_{it}^{(p)}$ is isometric on $L^p$ (cf.~\eqref{eq:pi-p-preserves-Lp-norm}), mixing, and $W^{\log,p}$ is compactly embedded in $L^p$, as detailed in the Propositions below.
We discuss the case $p=1$ in Section \ref{sec:L1-case}.
We also note that while $\pi^{(p)}_{it}$ is isometric on $L^p$, the representation grows exponentially when considered on $L^r$ for $r\not=p$, see Proposition \ref{prop:pi-p-on-L-q-expo-growth}.

\begin{proposition}\label{prop:mixing-in-intro}
	Under Assumption \ref{ass:3}, for any $1< p<\infty$ and $t\in\mathbb{R}$, the representation $\pi^{(p)}_{it}$ is mixing on $L^p(\partial X)$.
	The representation $\pi^{(1)}_{0}$ is weakly mixing on the subspace $L^1_0$ of $L^1(\partial X)$ consisting of functions with zero average.
\end{proposition}

\begin{proposition}\label{prop:main-cpct-embedding}
	In the setting of Assumption \ref{ass:3}, let $1<p<\infty$. Then $W^{\log,p}(Z)$ is compactly embedded in $L^p(Z)$.
\end{proposition}

In fact this is true more generally, Assumption \ref{ass:1} below is sufficient, see Section \ref{sec:compact-embedding}.

\medskip

Finally, as we require this for the proof of Theorem \ref{thm:main-s=0}, we prove a logarithmic Sobolev inequality in general Ahlfors--David regular spaces, cf.~Theorem \ref{thm:log-sobolev-ineq}. This is perhaps known to the experts in the area, but we were unable to find a proof in the literature in this generality, and thus we provide a proof.

\subsection{Acknowledgements}
The authors were supported (JS partially) by UK's Engineering and Physical Sciences Research Council (EPSRC) Standard Grant EP/V002899/1.

\section{Preliminaries}\label{sec:preliminaries}

\subsection{Setup}\label{subsec:setup}
The results in this paper hold in varying degree of generality, so let us record the assumptions here, in decreasing strength. The strongest setup is that of the above Assumption \ref{ass:3}; however in some Sections we in fact do not need the "inside" (i.e. the space $(X,\rho)$), but only a metric measure space $Z$ thought of as "a boundary" (i.e. $Z=\partial X$ in \ref{ass:3}).

\begin{assumption-w-label}{(Haus)}\label{ass:2}
	Let $(Z,d)$ be a bounded metric space of Hausdorff dimension $D$, such that its $D$-dimensional Hausdorff measure $\nu$ is Ahlfors--David $D$-regular.
\end{assumption-w-label}

This is provides sufficient link between the measure and the metric to perform some calculations involving the M\"obius self-maps of $Z$. (For example, the uniform boundedness results of \cite{Boucher-Spakula:2023} apply in this generality.) Let us recall the formulas we need.

Recall that a M\"obius map is a function $g:(Z,d)\to (Y,d')$ between metric space which preserves cross-ratios, i.e.~the quantity $[x,y;z,w] = \frac{d(x,z)d(y,w)}{d(x,w)d(y,z)}$, whenever this expression makes sense.
We denote by $\Mob(Z)$ the group of self-homeomorphisms of $(Z,d)$. For any $g\in \Mob(Z)$ there exists a \emph{metric derivative} of $g$, that is, a Lipschitz function $\D{g}:Z\to(0,\infty)$ satisfying the \emph{geometric mean value property}
\begin{equation}\label{eq:GMV}
	d^2(g\xi,g\eta)=\D{g}(\xi)\,\D{g}(\eta)\, d^2(\xi,\eta)\qquad \text{for all }g\in\Mob(Z),\, \xi,\eta\in Z.
\end{equation}
see e.g.~\cite[Lemmas 6, 7]{Nica:2013} (cf.~also \cite{Sullivan:1979}). Now, if $\nu$ is the $D$-dimensional Hausdorff measure which is also Ahlfors regular, $g$ preserves the class of $\nu$  and (a power of) $\D{g}$ gives the Radon--Nikodym derivative of the action (see e.g.~\cite[Lemma 8]{Nica:2013}):
\begin{equation}\label{eq:change-of-var}
	\mathrm{d}\nu(g\xi) = \D{g}^D(\xi)\,\mathrm{d}\nu(\xi)
	\qquad\text{for all }g\in\Mob(Z)\text{ and }\xi\in Z.
\end{equation}
We further have a cocycle formula
\begin{equation}\label{eq:metric-deriv-cocycle}
	\Db{gh}(\xi) = \D{g}(h\xi)\, \D{h}(\xi)\qquad\text{for all }g,h\in\Mob(Z)\text{ and }\xi\in Z.
\end{equation}

Finally, the function spaces, and the various Sobolev inequalities, work even more generally:

\begin{assumption-w-label}{(Ahlf)}\label{ass:1}
	Let $(Z,d,\nu)$ be a bounded metric measure space, which is Ahlfors--David $D$-regular (for some $D\geq0$).
\end{assumption-w-label}

\subsection{Notation}\label{subsec:notation}
We write $A\prec B$ when $A\leq cB$ for some constant $c>0$ which depends only on the underlying space, and parameters $p$ and $t$. If we want to emphasise the dependence on a particular parameter, we may write $A\prec_t B$. Furthermore, we write $A\asymp B$ iff $A\prec B$ and $B\prec A$.

\subsection{Strong hyperbolicity}\label{subsec:strong-hyperbolicity}
Let us return to the setting of Assumption \ref{ass:3} and be more precise. Let $(X,\rho)$ be a metric space. It is \emph{roughly geodesic} if there exists $C\geq0$ such that for all $x,y\in X$ there exists a function $\gamma: [a,b]\to X$ with $\gamma(a)=x$, $\gamma(b)=y$, and for all $s,s'\in[a,b]$ one has $|s-s'|-C\leq\rho(\gamma(s),\gamma(s'))\leq |s-s'|+C$.

A metric space $(X,\rho)$ is \emph{strongly hyperbolic} with parameter $\eps>0$ if $\exp(-\eps\langle x,y\rangle_o)\leq \exp(-\eps\langle x,z\rangle_o) + \exp(-\eps\langle z,y\rangle_o)$ for all $x,y,z,o\in X$ \cite[Definition 4.1]{Nica-Spakula:2016}. Fix $o\in X$. As mentioned above, then $\langle\cdot,\cdot\rangle_o$ extends continuously to $X\cup\partial X$, and $d(\xi,\eta)=\exp(-\eps\langle\xi,\eta\rangle_o)$ defines a metric on $\partial X$ \cite[Theorem 4.2]{Nica-Spakula:2016}.

In this setting, we have a formula for $\D{g}$. First the Busemann function: for $\xi\in\partial X$, $x,o\in X$, denote $\beta(\xi,x,o) = -2\langle \xi,x\rangle_o+\rho(x,o)$ (secretly thinking it "equals $\rho(\xi,x)-\rho(\xi,o)$"). Then for $g\in\Gamma$ we have $\D{g}(\xi)=\exp(-\eps\beta(\xi,go,o)) = \exp(2\eps\langle\xi,go\rangle_o-\eps\rho(go,o))$ \cite[Section 6.2]{Nica-Spakula:2016}.

We record a fact well--known to experts, and for the sake of completeness also a short proof.
\begin{proposition}\label{prop:extend-pt-to-boundary}
	Suppose \ref{ass:3}. Then there exist $M\geq0$ such that for any $g\in\Gamma$ there exists $\xi\in\partial X$ such that $|g|-M\leq \langle \xi,go\rangle_o\leq |g|$.
\end{proposition}

\begin{proof}
	The second inequality (for any $\xi\in\partial X$) follows immediately from the definition of Gromov product.

	For the first inequality, fix $\eta_1,\eta_2\in \partial X$, $\eta_1\not=\eta_2$; this is possible since the action of $\Gamma$ is assumed to be non-elementary.
	A direct calculation shows that for any $g\in\Gamma, \omega\in X$, we have $\langle go, \omega\rangle_o + \langle g^{-1}o,g^{-1}\omega\rangle_o = |g|$, and by continuity of the Gromov product on $X\cup \partial X$ this holds for $\omega\in\partial X$ as well.
	Hence
	\begin{equation*}
		\max(\langle go,g\eta_1\rangle_o, \langle go,g\eta_2\rangle_o)
		= |g| - \min(\langle g^{-1}o,\eta_1\rangle_o, \langle g^{-1}o,\eta_2\rangle_o)
		\geq |g|-\langle \eta_1,\eta_2 \rangle_o - \delta,
	\end{equation*}
	where $\delta\geq0$ is the hyperbolicity constant. Hence, choosing $\xi = g\eta_i$ which realises $\max_{i=1,2}\langle go,g\eta_i\rangle_o$, the claim follows with $M=\langle \eta_1,\eta_2\rangle_o+\delta$.
\end{proof}

\subsection{Zygmund/Orlicz spaces}\label{subsec:orlicz-spaces}

We shall need Zygmund spaces $L^p\log L$ ($p\geq 1$) and $L_{\exp}$; they are also instances of more general Orlicz spaces. We recall the definitions and some facts about them. Let $(Z,\nu)$ be a measure space, with $\nu(Z)<\infty$.

For $p\geq1$, denote $\Phi_{p}(t)= |t|^p\log^{+}|t|^p$ for $t\in\mathbb{R}$, where $\log^{+}(s)=\max(0,\log_e(s))$, $\log^{+}(0)=0$. Also denote $\Psi(t)=|t|$ for $t\in(-1,1)$ and $\Psi(t)=\exp(|t|-1)$ for $|t|\geq 1$. These are examples of the so-called Young functions in the theory of Orlicz spaces. Let $\Phi$ be one of the functions defined in this paragraph.

The \emph{Orlicz space} $L_\Phi(\nu)$ is the space of (a.e.--equivalence classes of) measurable functions $f:Z\to\mathbb{C}$, such that $\int_Z \Phi(|f|/k)\,\mathrm{d}\nu<\infty$ for some $k>0$ \cite[III]{Rao-Ren:1991}. When $\Phi=\Phi_1$ above, the resulting space is also referred to as a Zygmund space $L\log L$; we shall use this notation for it. When $\Phi=\Phi_{p}$, we will use the notation $L^p\log L$, and when $\Phi=\Psi$ we write $L_{\exp}$; the measure space $(Z,\nu)$ being understood from the context.

There are several ways to define (equivalent) norms on $L_\Phi(\nu)$, we shall use the Luxemburg norm \cite[Theorem III.2.3]{Rao-Ren:1991}: for $f\in L_\Phi(\nu)$, we put
\begin{equation}\label{eq:Luxemburg-norm}
	\|f|L_\Phi\| = \inf\big\{k\geq 0 \mid \int\nolimits_{Z}\Phi(|f|/k)\,\mathrm{d}\nu \leq 1\big\}.
\end{equation}
An immediate consequence of this definition is that if $f\in L^p\log L$, then $f^p\in L\log L$ and
\begin{equation}\label{eq:f^p-in-LlogL-implies-f-in-LplogL}
	\|f^p | L\log L\| = \|f | L^p\log L\|^p.
\end{equation}
Note that if $\nu(Z)<\infty$ and $p>1$, then $L^\infty \subseteq L_{\exp} \subseteq L^p\log L \subseteq L^p \subseteq L\log L  \subseteq L^1$. For example, the first inclusion when $\nu(Z)=1$ follows by observing that $\int_Z \Psi(|f|/k)\,\mathrm{d}\nu \leq \Psi\!\left(\|f|L^\infty\|/k\right)$. Since $\Psi(1)=1$, we obtain that $\|f|L_{\exp}\|\leq \|f|L^\infty\|$.

Finally, $\Phi_1$ and $\Psi$ form a so-called (normalised) \emph{complementary Young pair} \cite[I.3 and II.4]{Rao-Ren:1991}, and hence a generalised H\"older's inequality holds \cite[Proposition III.3.1]{Rao-Ren:1991}:
\begin{proposition}\label{prop:Young-for-LlogL-expL}
	If $f\in L\log L$ and $g\in L_{\exp}$, then
	\begin{equation}\label{eq:Young-for-LlogL-expL}
		\int\nolimits_Z |fg|\,\mathrm{d}\nu \leq \|f|L\log L\|\cdot \|g|L_{\exp}\|.
	\end{equation}
\end{proposition}

\section{Basic calculation}

In this section we go through a calculation which shows how to reduce bounding the representation $\pi^{(p)}_{it}$ to having an appropriate Sobolev inequality and bounding certain multipliers. (Somewhat similar calculations has been used before; the earliest place the authors are aware of is \cite{Mantero-Zappa:1983jfa}.)

\begin{proposition}\label{prop:basic-calculation}
	In the setting of Assumption \ref{ass:2}, let $p\geq1$, $t\in\mathbb{R}$, $\phi \in W^{\log,p}(Z)$ and $g\in \Mob(Z)$. Then
	\begin{equation}\label{eq:basic-calculation}
		\mathcal{E}_{\log,p}\big(\pi^{(p)}_{it}(g)\phi\big) \leq C_1\cdot\mathcal{E}_{\log,p}(\phi) + C_2\cdot\left\| \Theta_{g,t} | L_{\exp}\right\|\cdot\left\|\phi|L^p\log L\right\|^p,
	\end{equation}
	where
	\begin{equation}\label{eq:multipliers-defn}
		\Theta_{g,t}(\omega) = \int_{Z}\frac{\left|\frac{\D{g}^{D/(2p)+it}(\zeta)}{\D{g}^{D/(2p)+it}(\omega)}-1\right|^p}{d^D(\omega,\zeta)}\, \mathrm{d}\nu(\zeta).
	\end{equation}
\end{proposition}

To proceed from this result towards a norm bound for the representation, the $L^p\log L$ term will be controlled by $\|\phi|W^{\log,p}\|^p$ using the logarithmic Sobolev inequality (Corollary \ref{cor:Wsp-in-LlogL}). We refer to the control on the $L_{\exp}$ term as "geometric control", as it requires knowing something about the behaviour of metric derivatives $\D{g}$.
In the setting of Assumption \ref{ass:3}, we obtain geometric control which is linear in $\rho(go,o)$ in Section \ref{sec:geometric-control}, Corollary \ref{coro:bound-on-multipliers-2}.

\begin{remark}
	For comparison, a very much analogous calculation can be performed to show that $\pi_{s+it}^{(p)}$ is uniformly bounded on fractional Sobolev spaces $W^{s,p}(Z)$ (for $sp<D$, $0<s<1$), cf.~\cite[Section 7]{Boucher-Spakula:2023}. In that context, for the final steps one uses a dual pair of Lorentz spaces $L(\frac{D}{sp},\infty)$ and $L(\frac{D}{D-sp},1)$, fractional Sobolev inequality, and can establish a uniform geometric control.
\end{remark}

\begin{proof}
	We calculate:
\begin{align*}
	\mathcal{E}_{\log,p}\big(\pi^{(p)}_{it}(g)\phi\big)
	&=\iint_{Z^2}\frac{|\pi^{(p)}_{it}(g)\phi(\xi)-\pi^{(p)}_{it}(g)\phi(\eta)|^p}{d^D(\xi,\eta)} \,\mathrm{d}\nu(\xi) \mathrm{d}\nu(\eta)\\
	&=\iint_{Z^2}\frac{\big|\Db{g^{-1}}^{\frac{D}{p}-it}(\xi)\phi(g^{-1}\xi)-\Db{g^{-1}}^{\frac{D}{p}-it}(\eta)\phi(g^{-1}\eta)\big|^p}{d^D(\xi,\eta)} \,\mathrm{d}\nu(\xi) \mathrm{d}\nu(\eta)\\
	&=\iint_{Z^2}\frac{\big|\D{g}^{-\frac{D}{p}+it}(\omega)\phi(\omega)-\D{g}^{-\frac{D}{p}+it}(\zeta)\phi(\zeta)\big|^p}{d^D(g\omega,g\zeta)}\D{g}^{D}(\omega)\D{g}^{D}(\zeta)\mathrm{d}\nu(\omega) \mathrm{d}\nu(\zeta)\\
	&=\iint_{Z^2}\frac{\big|\D{g}^{-\frac{D}{p}+it}(\omega)\phi(\omega)-\D{g}^{-\frac{D}{p}+it}(\zeta)\phi(\zeta)\big|^p}{d^D(\omega,\zeta)}\D{g}^{\frac{D}{2}}(\omega)\D{g}^{\frac{D}{2}}(\zeta)\mathrm{d}\nu(\omega) \mathrm{d}\nu(\zeta)\\
	&=\iint_{Z^2}\frac{\left|\frac{\D{g}^{D/(2p)+it}(\zeta)}{\D{g}^{D/(2p)+it}(\omega)}\phi(\omega)-\frac{\D{g}^{D/(2p)+it}(\omega)}{\D{g}^{D/(2p)+it}(\zeta)}\phi(\zeta)\right|^p}{d^D(\omega,\zeta)}\,\mathrm{d}\nu(\omega) \mathrm{d}\nu(\zeta).\\
	\intertext{Using that $|x+y+z|^p\leq 3^{p-1}(|x|^p+|y|^p+|z|^p)$, we continue:}
	&\leq 3^{p-1}\cdot \mathcal{E}_{\log,p}(\phi)
	+2\cdot 3^{p-1}\int_{Z}\int_{Z}\frac{\left|\frac{\D{g}^{D/(2p)+it}(\zeta)}{\D{g}^{D/(2p)+it}(\omega)}-1\right|^p}{d^D(\omega,\zeta)} \mathrm{d}\nu(\zeta)\cdot |\phi(\omega)|^p \mathrm{d}\nu(\omega).
\end{align*}
For the remaining integral, we use the notation $\Theta_{g,t}$ from \eqref{eq:multipliers-defn}, and use H\"older's inequality \eqref{eq:Young-for-LlogL-expL}:
\begin{align*}
	\int_{Z}\int_{Z}\frac{\left|\frac{\D{g}^{D/(2p)+it}(\zeta)}{\D{g}^{D/(2p)+it}(\omega)}-1\right|^p}{d^D(\omega,\zeta)} & \mathrm{d}\nu(\zeta) \cdot |\phi(\omega)|^p \mathrm{d}\nu(\omega)
	= \int_{Z}\Theta_{g,t}(\omega) |\phi(\omega)|^p \mathrm{d}\nu(\omega)
	\\
	&\leq \left\| \Theta_{g,t} | L_{\exp}\right\|\cdot\left\||\phi|^p|L\log L\right\|\\
	&= \left\| \Theta_{g,t} | L_{\exp}\right\|\cdot\left\|\phi|L^p\log L\right\|^p.
\end{align*}
This finishes the proof.
\end{proof}

\section{Logarithmic Sobolev inequality}\label{sec:log-sobolev-ineq}

The results in this Section hold under Assumption \ref{ass:1}, i.e. in bounded Ahlfors--David regular metric measure spaces.

\begin{theorem}\label{thm:log-sobolev-ineq}
	Let $(Z,d,\nu)$ be a bounded Ahlfors--David $D$-regular metric measure space, and let $p\geq 1$. Then for any measurable function $f:Z\to\mathbb{C}$ we have
	\begin{equation}\label{eq:log-sobolev-ineq}
		\int |f|^p\log^+\frac{|f|^p}{\|f|L^p\|^p} \mathrm{d}\nu \prec \left\|f | L^p\right\|^p + \iint_{Z^2} \frac{\left|f(\xi)-f(\eta)\right|^p}{d^{D}(\xi,\eta)}\, \mathrm{d}\nu(\xi)\mathrm{d}\nu(\eta).
	\end{equation}
\end{theorem}

By the definition of Luxemburg norm on the space $L^p\log L$ (and using $\|f|L^p\|\leq\|f|W^{\log,p}\|$), we immediately obtain:

\begin{corollary}[The logarithmic Sobolev inequality]\label{cor:Wsp-in-LlogL}
	Under the assumptions of Theorem \ref{thm:log-sobolev-ineq}, we have
	\begin{equation*}
		\|f|L^p\log L\| \prec \|f|W^{\log,p}\|.
	\end{equation*}
\end{corollary}

To prove Theorem \ref{thm:log-sobolev-ineq}, we start with an analogue of \cite[Lemma 6.1]{Sobolev-Hitchhikers}.
\begin{lemma}\label{lem:hitchhikers-lemma-6.1}
	Let $(Z,d,\nu)$ be a bounded Ahlfors--David $D$-regular metric measure space. Then for any $E\subset Z$ and any $\xi\in Z$, we have
	\[
	\int_{Z\setminus E} \frac{\mathrm{d}\nu(\eta)}{d^{D}(\xi,\eta)}
	\geq
	-c\log(\nu(E)) + cD\log(\diam(Z)) + c\log(c),
	\]
	where $0<c\leq1$ is the lower regularity constant, i.e.~$cr^D\leq \nu(B(\cdot,r))$ for all $0\leq r\leq\diam(Z)$.
\end{lemma}

\begin{proof}
    With $\alpha=c^{-1/D}$, we have $\nu(B(\xi,\ga r))\geq c\alpha^D r^D= r^D$.
	Denote $B=B(\xi,\ga r)$ with $r=\nu(E)^\frac{1}{D}$. Then $\nu(B) \geq \nu(E)$, and so
	\begin{align*}
		\nu\left((Z\setminus E)\cap B\right)&=\nu(B)-\nu(E\cap B)
		\geq \nu(E)-\nu(E\cap B)
		= \nu\left(E\cap (Z\setminus B)\right).
	\end{align*}
	It follows that
	\begin{align*}
		\int_{Z\setminus E}\frac{\mathrm{d}\nu(\eta)}{ d^{D}(\xi,\eta)}
		&=\int_{(Z\setminus E)\cap B}\frac{\mathrm{d}\nu(\eta)}{ d^{D}(\xi,\eta)} + \int_{(Z\setminus E)\cap (Z\setminus B)}\frac{\mathrm{d}\nu(\eta)}{ d^{D}(\xi,\eta)}\\
		&\geq (\alpha r)^{-D}\nu\left((Z\setminus E)\cap B\right) + \int_{(Z\setminus E)\cap (Z\setminus B)}\frac{\mathrm{d}\nu(\eta)}{ d^{D}(\xi,\eta)}\\
		&\geq (\alpha r)^{-D}\nu\left(E\cap (Z\setminus B)\right) + \int_{(Z\setminus E)\cap (Z\setminus B)}\frac{\mathrm{d}\nu(\eta)}{ d^{D}(\xi,\eta)}\\
		&\geq \int_{E\cap (Z\setminus B)}\frac{\mathrm{d}\nu(\eta)}{ d^{D}(\xi,\eta)} + \int_{(Z\setminus E)\cap (Z\setminus B)}\frac{\mathrm{d}\nu(\eta)}{ d^{D}(\xi,\eta)}\\
		&=\int_{Z\setminus B}\frac{\mathrm{d}\nu(\eta)}{ d^{D}(\xi,\eta)},
	\intertext{with Lemma \ref{lem:sobolev-for-a-ball} below, we continue:}
		&\geq cD\left(\log(\diam(Z))-\log(\alpha r)\right)\\
		&=cD\log(\diam(Z)) - c\log(\nu(E)) - cD\log(c^{-1/D}). \qedhere
	\end{align*}
\end{proof}

\begin{lemma}\label{lem:sobolev-for-a-ball}
	Let $(Z,d,\nu)$ be a bounded Ahlfors--David $D$-regular metric measure space. Denote $0<c\leq 1\leq C$ to be the regularity constants, i.e. $cr^{D}\leq \nu(B(\cdot,r)) \leq Cr^D$ for $0\leq r\leq\diam(Z)$. Then for any $\xi\in Z$ and $0\leq r\leq\diam(Z)$, we have
	\begin{equation*}
		cD\log\left(\frac{\diam(Z)}{r}\right)
		\leq
		\int_{Z\setminus B(\xi,r)}\frac{\mathrm{d}\nu(\eta)}{d^{D}(\xi,\eta)}
		\leq
		C\left(1+ D\log\left(\frac{\diam(Z)}{r}\right)\right).
	\end{equation*}
\end{lemma}

\begin{proof}
Using the layer cake formula, we have
\begin{align*}
	\int_{Z\setminus B(\xi,r)}\frac{\mathrm{d}\nu(\eta)}{d^{D}(\xi,\eta)}
	&=\int\nolimits_{\diam(Z)^{-D}}^{r^{-D}} \nu\!\left(\{d^{-D}(\xi,\cdot) >t \} \right)\mathrm{d}t
	+\int\nolimits_{0}^{\diam(Z)^{-D}}  \nu\!\left(Z\setminus B(\xi,r)\right)\mathrm{d}t.
\end{align*}
The latter integral equals
\begin{equation*}
\diam(Z)^{-D}\cdot \nu(Z\setminus B(\xi,r))\leq
\diam(Z)^{-D}\cdot \nu(Z) \leq C.
\end{equation*}
For the former, we calculate:
\begin{align*}
	\int\nolimits_{\diam(Z)^{-D}}^{r^{-D}} \nu\!\left(\{d^{-D}(\xi,\cdot) >t \} \right)\mathrm{d}t
	&=\int\nolimits_{\diam(Z)^{-D}}^{r^{-D}} \nu\!\left(\{d(\xi,\cdot) < t^{-1/D} \} \right)\mathrm{d}t\\
	&\asymp \int\nolimits_{\diam(Z)^{-D}}^{r^{-D}} t^{-1}\mathrm{d}t 
	= \left[ \log(t) \right]_{\diam(Z)^{-D}}^{r^{-D}} \\
	&=\log(\diam(Z)^D)-\log(r^D)=D\log(\diam(Z)/r),
\end{align*}
where $\asymp$ uses Ahlfors regularity.
The claim now follows.
\end{proof}

\begin{proof}[Proof of Theorem \ref{thm:log-sobolev-ineq}]
	We start with a series of reductions.
	We assume that $f\in L^p$ and $\mathcal{E}_{\log,p}(f)<\infty$, otherwise the inequality holds trivially.

	By adjusting the asymptotic constant in the inequality, we can assume that $\nu(Z)=\diam(Z)=1$.

	The inequality doesn't change if we multiply $f$ by a scalar, and holds when $f\equiv 0$, so we may assume $\|f|L^p\|=1$.

	As $\left| |x|-|y|\right| \leq |x-y|$ for any $x,y\in\mathbb{C}$, it is sufficient to prove the inequality for $|f|$, and thus we can assume that $f\geq 0$.

	Finally, we claim that it suffices to prove the inequality
	\begin{equation}\label{localeq:proof-of-log-sobolev-g}
		\int |g|^p\log^+|g|^p \mathrm{d}\nu \prec
		1 + \mathcal{E}_{\log,p}(g).
	\end{equation}
	for $g=\max(f,1)$. Indeed, the values of the left-hand sides of \eqref{eq:log-sobolev-ineq} and \eqref{localeq:proof-of-log-sobolev-g} are the same, and $\mathcal{E}_{\log,p}(g)\leq \mathcal{E}_{\log,p}(f)$, since if $0\leq f(x)\leq 1$ and $f(y)\geq 1$, then
	\begin{equation*}
		|f(y)-f(x)| = f(y)-f(x) \geq f(y) - 1 = |g(y)-g(x)|.
	\end{equation*}
	Observe that $\|g|L^p\|\leq 2^{1/p}$.

	Summarizing, for the rest of the proof we assume that $f:Z\to [1,\infty)$ and $1\leq \|f|L^p\|^p\leq 2$, and we aim to prove \eqref{eq:log-sobolev-ineq} for such $f$.

	For the bulk of the argument, we closely follow the proof of the standard fractional Sobolev inequality in \cite[Section 6]{Sobolev-Hitchhikers}. We use the same notation for easier comparison.

	We discretise: for $k\in\mathbb{Z}$ we denote $A_k=\{|f|>2^k\}$ and $a_k = \nu(A_k)$. We further subdivide into $D_k=\{2^{k+1}\geq|f|> 2^{k}\}$ and denote $d_k=\nu(D_k)$. Note that since $\nu(Z)=1$, we have $a_k^{-1}\geq 1$ and thus $\log(a_k^{-1})\geq0$ for all $k\in\mathbb{Z}$ such that $a_k\not=0$. Furthermore $a_k=1$ for all $k<0$ by our assumption on $f$, and $a_{k}\geq a_{k+1}$ for all $k$.

	Observe that
	\begin{align}\label{localeq:sum-bound-for-Hitchhiker-L6.3}
		\sum_{i\in\mathbb{Z}} d_i2^{pi}
		&\leq \sum_{i\in\mathbb{Z}} \int_{D_i} |f(\xi)|^p \mathrm{d}\nu(\xi)
		= \left\|f|L^p\right\|^p \leq 2.
	\end{align}
	In fact $\sum_{i\in\mathbb{Z}} d_i2^{pi} \asymp \text{const}$, but we shall not need the other bound.
	
	For $(x,y)\in D_i\times D_j$ with $j\leq i-2$ one has $|f(x)-f(y)|\geq 2^{i-1}$, and thus
	\begin{align}
		\sum_{\fratop{j\in\mathbb{Z}}{j\leq i-2}}\int_{D_j}\frac{|f(x)-f(y)|^p}{d(x,y)^{D}}\mathrm{d}\nu(y)
		&\geq 2^{\fratop{p(i-1)}\sum_{j\in\mathbb{Z}}{j\leq i-2}} \int_{D_j}\frac{\mathrm{d}\nu(y)}{d(x,y)^{D}} \notag\\
		&= 2^{p(i-1)}\int_{Z\setminus A_{i-1}}\frac{\mathrm{d}\nu(y)}{d(x,y)^{D}} \notag\\
		&\geq c_0 2^{pi}\log(a_{i-1}^{-1}) - c_1 2^{pi} \label{localeq:Hitchhiker-pre6.15}
	\end{align}
	for suitable constants $c_0,c_1>0$ by Lemma \ref{lem:hitchhikers-lemma-6.1}. Now we execute the main scheme of the proof of \cite[Lemma 6.3]{Sobolev-Hitchhikers}: obtain the analogues of \cite[(6.15)--(6.17)]{Sobolev-Hitchhikers} using our \eqref{localeq:Hitchhiker-pre6.15}.
	For any $i\in\mathbb{Z}$ we thus have
	\begin{equation}\label{localeq:Hitchhiker-6.15}
		\sum_{\fratop{j\in\mathbb{Z}}{j\leq i-2}}\iint_{D_i\times D_j}\frac{|f(x)-f(y)|^p}{d(x,y)^{D}}\mathrm{d}\nu(x)\mathrm{d}\nu(y)
		\geq c_0 2^{pi}\log(a_{i-1}^{-1})d_i - c_12^{pi}d_i
	\end{equation}	
	Denoting $S=\sum_{i\in\mathbb{Z}, a_{i-1}\not=0}2^{pi}\log(a_{i-1}^{-1})d_{i}$, summing over $i$, and using \eqref{localeq:sum-bound-for-Hitchhiker-L6.3}, we further get
	\begin{equation}\label{localeq:Hitchhiker-6.17}
		\sum_{\fratop{i\in\mathbb{Z}}{a_{i-1}\not=0}}\sum_{\fratop{j\in\mathbb{Z}}{j\leq i-2}}\iint_{D_i\times D_j} \frac{|f(x)-f(y)|^p}{d(x,y)^{D}}\mathrm{d}\nu(x)\mathrm{d}\nu(y)
		\geq c_0 S - 2c_1.
	\end{equation}
	Coming back to \eqref{localeq:Hitchhiker-6.15} and noting that
	$d_i=a_i-\sum_{\ell\in\mathbb{Z},\, \ell\geq i+1} d_{\ell}$, we also get
	\begin{multline}\label{localeq:Hitchhiker-6.16}
		\sum_{\fratop{j\in\mathbb{Z}}{j\leq i-2}}\iint_{D_i\times D_j}\frac{|f(x)-f(y)|^p}{d(x,y)^{D}}\mathrm{d}\nu(x)\mathrm{d}\nu(y)
		\\\geq c_0\left(2^{pi}\log(a_{i-1}^{-1})a_i-\textstyle\sum_{\ell\in\mathbb{Z},\,\ell\geq i+1} 2^{pi} \log(a_{i-1}^{-1})d_{\ell} \right) -c_12^{pi}d_i.
	\end{multline}
	Running the same calculation as in \cite[(6.14)]{Sobolev-Hitchhikers} yields
	\begin{equation*}
		S\geq \sum_{i\in\mathbb{Z},\, a_{i-1}\not=0}\sum_{\ell\in\mathbb{Z},\,\ell\geq i+1} 2^{pi}\log(a_{i-1}^{-1})d_{\ell}.
	\end{equation*}
	Using this, \eqref{localeq:sum-bound-for-Hitchhiker-L6.3}, and \eqref{localeq:Hitchhiker-6.17} in what follows, we start by summing over $i$ in \eqref{localeq:Hitchhiker-6.16}:
	\begin{align*}
		\sum_{\fratop{i\in\mathbb{Z}}{a_{i-1}\not=0}}&\sum_{\fratop{j\in\mathbb{Z}}{j\leq i-2}}\iint_{D_i\times D_j} \frac{|f(x)-f(y)|^p}{d(x,y)^{D}}\mathrm{d}\nu(x)\mathrm{d}\nu(y)
		\\
		&\geq c_0\left[\sum_{\fratop{i\in\mathbb{Z}}{a_{i-1}\not=0}}2^{pi}\log(a_{i-1}^{-1})a_i-\sum_{\fratop{i\in\mathbb{Z}}{a_{i-1}\not=0}}\sum_{\fratop{\ell\in\mathbb{Z}}{\ell\geq i+1}} 2^{pi} \log(a_{i-1}^{-1})d_{\ell} \right] -c_1\!\!\!\sum_{\fratop{i\in\mathbb{Z}}{a_{i-1}\not=0}}\!\!\! 2^{pi}d_i\\
		&\geq c_0\left[\sum_{\fratop{i\in\mathbb{Z}}{a_{i-1}\not=0}}2^{pi}\log(a_{i-1}^{-1})a_i-S \right] - 2c_1 \\
		&\geq c_0\sum_{\fratop{i\in\mathbb{Z}}{a_{i-1}\not=0}}2^{pi}\log(a_{i-1}^{-1})a_i
		- \sum_{\fratop{i\in\mathbb{Z}}{a_{i-1}\not=0}}\sum_{\fratop{j\in\mathbb{Z}}{j\leq i-2}}\iint_{D_i\times D_j} \frac{|f(x)-f(y)|^p}{d(x,y)^{D}}\mathrm{d}\nu(x)\mathrm{d}\nu(y)\\
		&\hspace{1em} -4c_1.
	\end{align*}
	By moving the integral on the right-hand side to the left, we obtain
	\begin{equation*}
		\sum_{\fratop{i\in\mathbb{Z}}{a_{i-1}\not=0}}\sum_{\fratop{j\in\mathbb{Z}}{j\leq i-2}}\iint_{D_i\times D_j} \frac{|f(x)-f(y)|^p}{d(x,y)^{D}}\mathrm{d}\nu(x)\mathrm{d}\nu(y)
		\geq \frac{c_0}{2}\sum_{\fratop{i\in\mathbb{Z}}{a_{i-1}\not=0}}2^{pi}\log(a_{i-1}^{-1})a_i - 2 c_1.
	\end{equation*}
	As in \cite[(6.19)]{Sobolev-Hitchhikers}, by using symmetry and that the sets $D_i\times D_j$ partition $Z^2$, we get
	\begin{align}\label{localeq:end-of-L6.3-in-Hitchhikers}
		\iint_{Z^2} \frac{|f(x)-f(y)|^p}{d(x,y)^{D}}\mathrm{d}\nu(x)\mathrm{d}\nu(y)+4c_1
		&\geq c_0\sum_{\fratop{i\in\mathbb{Z}}{a_{i-1}\not=0}}2^{pi}\log(a_{i-1}^{-1})a_i \notag\\
		&= c_0\sum_{i=1}^{N} 2^{pi}\log(a_{i-1}^{-1})a_i \notag\\
		&\geq c_0\sum_{i=1}^{N} 2^{pi}\log(a_{i-1}^{-1})d_i,
	\end{align}
	where $N=\max\{i: a_{i-1}\not=0\} \in \mathbb{N}\cup\{\infty\}$, or equivalently, $N$ is the smallest integer (or $\infty$) such that $f\leq 2^N$ almost everywhere.
	To finish the proof, we depart from \cite{Sobolev-Hitchhikers}. By Markov's Inequality and \eqref{localeq:sum-bound-for-Hitchhiker-L6.3}, we have
	\begin{equation*}
		a_k = \nu\left(\{ |f|^p > 2^{pk} \}\right) \leq 2^{-pk} \|f|L^p\|^p \leq 2^{-pk+1}.
	\end{equation*}
	Consequently, if $a_{i-1}\not=0$, then
	\begin{equation*}
		\log(a_{i-1}^{-1}) \geq \log\!\left( 2^{p(i-1)-1} \right) = \log(2^{p(i+1)})-\log(2^{2p+1}).		
	\end{equation*}
	Thus we have
	\begin{align*}
		\sum_{i=1}^{N} 2^{pi}\log(a_{i-1}^{-1})d_i
		&\geq \sum_{i=1}^{N} 2^{pi} \log(2^{p(i+1)})d_i - \log(2^{2p+1}) \cdot \sum_{i=1}^{N}2^{pi}d_i\\
		&\geq 2^{-p}\sum_{i=1}^{N} 2^{p(i+1)} \log^+(2^{p(i+1)})d_i - 2\log(2^{2p+1})\\
		&\geq 2^{-p}\int |f|^p\log^+|f|^p\mathrm{d}\nu - 2\log(2^{2p+1}).
	\end{align*}
	This, together with \eqref{localeq:end-of-L6.3-in-Hitchhikers}, implies \eqref{localeq:proof-of-log-sobolev-g}, and we are done.
\end{proof}

\section{Geometric control}\label{sec:geometric-control}

In this section we estimate the multipliers \eqref{eq:multipliers-defn} from Proposition \ref{prop:basic-calculation}, and for this we need the "inside", i.e. the setting of Assumption \ref{ass:3}. In this setting, let us extend the notation for the metric $d$ as follows: if $x\in X$ and $\xi\in Z=\partial X$, let $d(x,\xi)=\exp(-\eps\langle x,\xi\rangle_o)$. As $X$ is strongly hyperbolic, this extension of $d$ still satisfies a "triangle inequality", meaning $d(x,\xi)\leq d(x,\eta)+d(\eta,\xi)$ for any $x\in X, \eta,\xi\in \partial X$.

Recall the formula \eqref{eq:multipliers-defn} for the multipliers that we estimate in this section: for $g\in\Gamma$, $t\in\mathbb{R}$, and $\omega\in Z$, we have $\Theta_{g,t}(\omega) = \int_{Z}d^{-D}(\omega,\zeta) \left|\frac{\D{g}^{D/(2p)+it}(\zeta)}{\D{g}^{D/(2p)+it}(\omega)}-1\right|^p \mathrm{d}\nu(\zeta)$. Using the formulas from \ref{subsec:strong-hyperbolicity}, we have
\begin{equation*}
\D{g}(\omega)=\exp(2\eps\langle\omega,go\rangle_o)\exp(-\eps\rho(go,o)) = d^{-2}(\omega,go)\exp(-\eps\rho(go,o)),
\end{equation*}
and hence we can rewrite
$\Theta_{g,t}(\omega) = \int_{Z}d^{-D}(\omega,\zeta) \left|\frac{d(\omega,go)^{D/p+2it}}{d(\zeta,go)^{D/p+2it}}-1\right|^p \mathrm{d}\nu(\zeta)$. There is nothing special about $go\in X$, so, abusing the notation somewhat, let us write for $a\in X$,
\begin{equation}\label{eq:multipliers-defn-2}
\Theta_{a,t}(\omega)=\int_{Z}d^{-D}(\omega,\zeta) \left|1-\left(\frac{d(\omega,a)}{d(\zeta,a)}\right)^{D/p+it}\right|^p \mathrm{d}\nu(\zeta).
\end{equation}

\begin{proposition}\label{prop:bound-on-multipliers-1}
	In the setting of Assumption \ref{ass:3} and notation \eqref{eq:multipliers-defn-2}, for any $t\in\mathbb{R}$, $a\in X$, and $\eta\in Z$ we have
	\begin{equation*}
		\Theta_{a,t}(\eta) \prec_t -\log(d(a,\eta))+\rho(a,o)+1.
	\end{equation*}
\end{proposition}

\begin{corollary}\label{coro:bound-on-multipliers-2}
	In the setting of Assumption \ref{ass:3} and notation \eqref{eq:multipliers-defn} from Proposition \ref{prop:basic-calculation}, for any $g\in\Gamma$ and $t\in\mathbb{R}$, we have
	\begin{equation*}
		\left\| \Theta_{g,t} | L_{\exp}\right\| \prec_t 1+\rho(go,o).
	\end{equation*}
\end{corollary}

\begin{proof}
	For any $g\in\Gamma$ and $\omega\in Z$ we have $0\leq \langle x,\omega\rangle_o\leq \rho(x,o)$, and thus $-\log(d(go,\omega))\leq \eps\rho(go,o)$. Hence by Proposition \ref{prop:bound-on-multipliers-1} we have $\|\Theta_{g,t}|L^\infty\|\leq (C_0\eps+C_1)\rho(go,o)+C_2$. The conclusion now follows since $L^\infty\subseteq L_{\exp}$ (as $\nu(Z)=1$ the constants even stay the same).
\end{proof}

We need some preparation for the proof of Proposition \ref{prop:bound-on-multipliers-1}. The following is \cite[Lemma 6.1]{Boucher-Spakula:2023} (the proof is a standard layer--cake calculation).

\begin{lemma}\label{lem:int-over-ball-calculation}
	Suppose \ref{ass:1}.
	Then for any $\alpha <D$, $\xi\in Z$, and $0\leq r\leq \diam(Z)$, we have
\begin{equation*}
	\int_{B(\xi,r)}\frac{\mathrm{d}\nu(\eta)}{d^{\alpha}(\xi,\eta)}\prec r^{D-\alpha}.
\end{equation*}
\end{lemma}

There is no bound when $\alpha=D$, however there is an estimate if we are in the setting of Assumption \ref{ass:3}, and consider $a\in X$ and $B(a,r)=\{\xi\in Z\mid d(a,x)\leq r\}$. First we need a lemma.

\begin{lemma}\label{lem:shadow-B}
	In the setting of Assumption \ref{ass:3}, for any $a\in X$, and $0\leq r\leq 1=\diam(Z)$, we have $\nu(B(a,r))\prec r^D$. For $0\leq r<\exp(-\eps\rho(a,o))$, $B(a,r)=\emptyset$.
\end{lemma}

\begin{proof}
	This follows from the fact that there exists $c\geq1$, such that for any $a\in X$ there exists $\theta\in\partial X$, such that $B(\theta,r/c)\subseteq B(a,r)\subseteq B(\theta,cr)$ whenever $B(a,r)\not=\emptyset$. (One can take $\theta$ in a "shadow" of $a$.) This, in turn, follows from approximating finite collections of points in $X$ by trees, cf.~for instance \cite[Theorem 2.4]{Blachere-Haissinsky-Mathieu:2011}.

	The second claim, observe that for any $\omega\in\partial X$,  $\langle a,\omega\rangle_o\leq \rho(a,o)$. Equivalently, $d(a,\omega)\geq \exp(-\eps\rho(a,o))$.
\end{proof}

\begin{lemma}\label{lem:int-over-ball-power-D}
	In the setting of Assumption \ref{ass:3}, for any $a\in X$, and $\exp(-\eps\rho(a,o))\leq r\leq 1=\diam(Z)$, we have
\begin{equation*}
	\int_{B(a,r)}\frac{\mathrm{d}\nu(\omega)}{d^{D}(a,\omega)}
	\prec
	1 + \eps\rho(a,o) + \log(r)
	\leq 1+\eps\rho(a,o).
\end{equation*}
	For $0\leq r <\exp(-\eps\rho(a,o))$, the integral is zero.
\end{lemma}

\begin{proof}
	Denote $|a|=\rho(a,o)$. Then for any $\omega\in Z$ we have $\exp(-\eps|a|)\leq d(a,\omega)$, or equivalently, $d^{-D}(a,\omega)\leq \exp(\eps D |a|)$.

	Otherwise, we use the layer cake formula to get
	\begin{align*}
		\int_{B(a,r)}\frac{\mathrm{d}\nu(\omega)}{d^{D}(a,\omega)}
		&= \int\nolimits_0^{r^{-D}} \nu\!\left(B(a,r)\right)\mathrm{d}t
		 + \int\nolimits_{r^{-D}}^{\exp(\eps D |a|)} \nu\!\left(\{d^{-D}(a,\cdot) >t \} \right)\mathrm{d}t.
	\end{align*}
	The former of the two integrals above equals $r^{-D} \nu(B(a,r)) \prec 1$ by Lemma \ref{lem:shadow-B}.
	For the latter, we calculate:
	\begin{align*}
		\int\nolimits_{r^{-D}}^{\exp(\eps D |a|)} \nu\left(\{d^{-D}(a,\cdot) >t \} \right)\mathrm{d}t
		&=\int\nolimits_{r^{-D}}^{\exp(\eps D |a|)} \nu\left(\{d(a,\cdot) < t^{-1/D} \} \right)\mathrm{d}t\\
		&\prec \int\nolimits_{r^{-D}}^{\exp(\eps D |a|)} t^{-1}\mathrm{d}t 
		=\left[\log(t)\right]_{r^{-D}}^{\exp(\eps D |a|)}\\
		&=
		D(\eps |a| +\log(r)).
	\end{align*}
	Note that $\eps|a|+\log(r) > 0$. The claim follows.
\end{proof}

\begin{proof}[Proof of Proposition \ref{prop:bound-on-multipliers-1}]
	We follow the same scheme as \cite[Section 6]{Boucher-Spakula:2023}.

	Fix $a\in X$ and $\eta\in Z$. Denote $|a|=\rho(a,o)$ and
	$$
	U(\eta)= B(\eta, d(a,\eta)/2), \quad U(a) = B(a, d(a,\eta)/2).
	$$
	The set $U(a)$ is empty when $d(a,\eta)/2<\exp(-\eps|a|)$, or equivalently when $|a| + \frac{1}{\eps}\log(2) < \langle a,\eta\rangle_o$.

	Consider $\xi\in U(a)$. Then $d(a,\xi)<\frac{1}{2}d(a,\eta) \leq \frac{1}{2}(d(a,\xi)+d(\xi,\eta))$, and hence $d(a,\xi)\leq d(\xi,\eta)$.  Next, we have $d(a,\eta) \leq d(a,\xi)+d(\xi,\eta)\leq 2 d(\xi,\eta)$. Conversely, $d(\xi,\eta) \leq d(a,\xi)+d(a,\eta)\leq \frac{3}{2}d(a,\eta)$. Summarising, we have
	\begin{equation}\label{localeq:Ua-updown-bounds}
		\xi\in U(a)\implies d(a,\xi)\prec d(a,\eta)\asymp d(\xi,\eta). %
	\end{equation}
	Symmetrically, we have
	\begin{equation}\label{localeq:Ueta-updown-bounds}
		\xi\in U(\eta)\implies d(\xi,\eta) \prec d(a,\eta)\asymp d(a,\xi). %
	\end{equation}
	Finally, in $W := Z\setminus (U(a)\cup U(\eta))$, a similar calculation with the triangle inequality gives
	\begin{equation}\label{localeq:ZminusUs-updown-bounds}
		d(a,\eta)\leq 2d(a,\xi)\asymp d(\xi,\eta). %
	\end{equation}
	For easy reference, recall that for any $x,y\in\mathbb{C}$ we have
	\begin{equation}\label{localeq:x+y^p}
		|x-y|^p\leq 2^{p-1}(|x|^p+|y|^p).
	\end{equation}
	We now estimate the integral in the statement of the Proposition on three subsets of $Z$.

	On $W$, we have that $\frac{d(a,\eta)}{d(a,\xi)}\leq 2$. We use \eqref{localeq:x+y^p} and Lemma \ref{lem:sobolev-for-a-ball} to estimate
	\begin{align*}
		\int_W\frac{\left|1-\big[\frac{ d(a,\eta)}{ d(a,\xi)}\big]^{D/p+it}\right|^p}{ d^{D}(\xi,\eta)}\mathrm{d}\nu(\xi)
		&\prec
		\int_{Z\setminus U(\eta)}\frac{1}{ d^{D}(\xi,\eta)}\mathrm{d}\nu(\xi)
		\prec 1-\log d(a,\eta).
	\end{align*}

	On $U(a)$, when nonempty, we start with \eqref{localeq:x+y^p} again, then use \eqref{localeq:Ua-updown-bounds}. For the last step we use Lemma \ref{lem:shadow-B} and Lemma \ref{lem:int-over-ball-power-D} to get
	\begin{align*}
		\int_{U(a)}\frac{\left|1-\big[\frac{ d(a,\eta)}{ d(a,\xi)}\big]^{D/p+it}\right|^p}{ d^{D}(\xi,\eta)}\mathrm{d}\nu(\xi)
		&\prec \frac{\nu(U(a))}{d^{D}(a,\eta)}
		+ \int_{U(a)}\frac{1}{d^{D}(a,\xi)} \mathrm{d}\nu(\xi)\\
		&\prec 1+|a|.
	\end{align*}
	
	Finally, we deal with $U(\eta)$. Here $\frac{ d(a,\eta)}{ d(a,\xi)} \in [\frac{2}{3},2]$, and so the Mean Value Theorem will imply the first $\prec$ in the calculation below. Next, the triangle inequality gives that
	\begin{equation*}%
		\left|1-\tfrac{ d(a,\xi)}{ d(a,\eta)}\right| \leq \tfrac{d(\xi,\eta)}{d(a,\eta)}.
	\end{equation*}
	This, followed by \eqref{localeq:Ueta-updown-bounds} and Lemma \ref{lem:int-over-ball-calculation} finish the estimate:
	\begin{align*}
		\int_{U(\eta)}\frac{\left|1-\big[\frac{ d(a,\eta)}{ d(a,\xi)}\big]^{D+it}\right|^p}{ d^{D}(\xi,\eta)}\mathrm{d}\nu(\xi)
		&\prec_{t}
		\int_{U(\eta)}\frac{\left|1-\frac{ d(a,\eta)}{ d(a,\xi)}\right|^p}{ d^{D}(\xi,\eta)}\mathrm{d}\nu(\xi)\\
		&\leq
		\int_{U(\eta)} \frac{\mathrm{d}\nu(\xi)}{ d^p(a,\xi)d^{D-p}(\xi,\eta)}
		\prec_{t}
		1.
		\qedhere
	\end{align*}	
\end{proof}

\section{Proof of Theorem \ref{thm:main-s=0}}\label{sec:proof-of-s=0}

In this Section, we finish the proof of Theorem \ref{thm:main-s=0}. As mentioned above, the basic calculation (Proposition \ref{prop:basic-calculation}), together with the logarithmic Sobolev inequality (Theorem \ref{thm:log-sobolev-ineq}) and geometric control (Corollary \ref{coro:bound-on-multipliers-2}) imply that for any $g\in\Gamma$, $\mathcal{E}_{\log,p}(\pi^{(p)}_{it}(g)\phi) \prec_t (1+\rho(go,o))\cdot\|\phi|W^{\log,p}\|^p$.
To finish the bound on $W^{\log,p}$, we observe that $\pi^{(p)}_{it}(g)$ preserves $L^p$-norm:
\begin{align}\label{eq:pi-p-preserves-Lp-norm}
	\|\pi^{(p)}_{it}(g)\phi|L^p\|^p &= \int_{Z} \Db{g^{-1}}^D(\xi)\, |\phi(g^{-1}\xi)|^p\,\mathrm{d}\nu(\xi)
	=\int_Z |\phi(\omega)|^p\,\mathrm{d}\nu(\omega)
	=\|\phi|L^p\|^p.
\end{align}

The final piece is to show the lower bound on $\|\pi_{it}^{(p)}(g)|W^{\log,p}\|$. The constant function $\mathbf{1}$ has norm $1$ in $W^{\log,p}$ (under Assumption \ref{ass:3}). Using the logarithmic Sobolev inequality (Corollary \ref{cor:Wsp-in-LlogL}), we have
\begin{align*}
	\|\pi_{it}^{(p)}(g)\mathbf{1} \mid W^{\log,p}\|^p \succ
	\|\pi_{it}^{(p)}(g)\mathbf{1} \mid L^p\log L\|^p
	= \|(\pi_{it}^{(p)}(g)\mathbf{1})^p \mid L\log L\|,
\end{align*}
and the desired conclusion now follows from the next Proposition.

\begin{proposition}\label{prop:LlogL-of-pi-g-1-lower-bound}
	In the setting of Assumption \ref{ass:3}, let $p\geq1$. Then we have $\|(\pi_{it}^{(p)}(g)\mathbf{1})^p\mid L\log L\| \succ \rho(go,o)$, for any $t\in\mathbb{R}$ and $g\in\Gamma$.
\end{proposition}

\begin{proof}
	Let $C>0$ be the constant from Lemma \ref{lem:bound-on-integral-of-Gromov-product} below, and suppose that $g\in\Gamma$ satisfies $\rho(go,o)>2C$. Then for any $0 < k < \frac{eD\eps}{e+1}(\rho(go,o)-2C)$ we have
	\begin{align*}
		\int_{Z}\Phi_1\!\left((\pi_{it}^{(p)}(g)\mathbf{1})^p(\xi)/k\right)\mathrm{d}\nu(\xi)
		&= \frac{1}{k}\int_{Z} \Db{g^{-1}}^D(\xi)\cdot\log^+\!\left(\Db{g^{-1}}^D(\xi)/k\right)\mathrm{d}\nu(\xi)\\
		&= \frac{1}{k}\int_{Z} \log^+\!\left(\Db{g^{-1}}^D(g\omega)/k\right) \mathrm{d}\nu(\omega)\\
		&= \frac{1}{k}\int_{Z} \log^+\!\left(\D{g}^{-D}(\omega)/k\right) \mathrm{d}\nu(\omega)\\
		&\geq \frac{1}{k} \int_{Z} \log\!\left(\D{g}^{-D}(\omega)/k\right) \mathrm{d}\nu(\omega)\\
		&= \frac{1}{k}\log\frac{1}{k} + \frac{D\eps}{k}\int_{Z} \rho(go,o) - 2\langle\omega,go\rangle_o \,\mathrm{d}\nu(\omega)\\
		&\geq -\frac{1}{e} + \frac{D\eps}{k}(\rho(go,o) - 2C) > 1.
	\end{align*}
	From the definition of Luxemburg norm, for $g$ as above, we now have that $\|(\pi_{it}^{(p)}(g)\mathbf{1})^p\mid L\log L\| \geq \frac{eD\eps}{e+1}(\rho(go,o)-2C)$. Consequently, $\|(\pi_{it}^{(p)}(g)\mathbf{1})^p\mid L\log L\|\succ \rho(go,o)$ for any $g\in\Gamma$.
\end{proof}

\begin{lemma}\label{lem:bound-on-integral-of-Gromov-product}
	In the setting of Assumption \ref{ass:3}, there exists $C>0$, such that for any $a\in X$, $\int_{Z} \langle \omega,a\rangle_o \,\mathrm{d}\nu(\omega) \leq C$.
\end{lemma}

\begin{proof}
	Since the action of $\Gamma$ on $X$ is cocompact, we can assume that $a=go$ for some $g\in\Gamma$. By Proposition \ref{prop:extend-pt-to-boundary}, there is $\xi\in Z$ with $|g|\leq \langle go,\xi\rangle_o +M$ for a spatial constant $M\geq 0$. By hyperbolicity, for any $\omega\in Z$, we have $\min(\langle \omega,go\rangle_o, \langle go,\xi\rangle_o) \leq \langle \omega, \xi\rangle_o+\delta$. Thus either $\langle \omega,go\rangle_o\leq \langle \omega, \xi\rangle_o+\delta$ or $\langle \omega,go\rangle_o \leq |g| \leq \langle go,\xi\rangle_o + M \leq \langle \omega, \xi\rangle_o+\delta+M$; in either case $\langle \omega,go\rangle_o\leq \langle \omega, \xi\rangle_o+\delta+M$.
	
	Hence it suffices to bound $\int_{Z}\langle \omega,\xi\rangle_o\,\mathrm{d}\nu(\omega)$ by a constant, independently of $\xi\in Z$. However $\alpha\eps\langle \omega,\xi\rangle_o \leq \exp(\alpha\eps\langle \omega,\xi\rangle_o) = d^{-\alpha}(\omega,\xi)$ for any $\alpha>0$, so we are done by Lemma \ref{lem:int-over-ball-calculation}.

	Alternatively, one can perform a direct calculation using the Shadow Lemma.
\end{proof}

\section{Compact embedding}\label{sec:compact-embedding}
 
We aim for showing that $W^{\log,p}(Z)$ is compactly embedded in $L^p(Z)$ for $p>1$. We will use compactness of integral kernel operators given by off-diagonal cut-off kernels $K_\eps = \chi_{\{(\xi,\omega)\in Z^2 | d(\xi,\omega)>\eps\}}\cdot d^{-D}$ for $\eps>0$ (Proposition \ref{prop:compactness-of-kernel-ops}). Proposition \ref{prop:compact-embedding-general} shows how can one use this to obtain a compact embedding result, but we need to set up some notation first.

The argument presented here has a similar structure as \cite[Subsection 3.1]{Gerontogiannis-Mesland:2023}, but in a more general setting.

Suppose that $(Z,\nu)$ is a probability measure space. For a measurable function $K:Z^2\to\mathbb{C}$, we denote the corresponding integral kernel operator by
\begin{equation*}
	T_K(\phi)(\xi) = \int_{Z} K(\xi,\omega)\phi(\omega)\,\mathrm{d}\nu(\omega).
\end{equation*}
Next, if $K\geq0$, for $1\leq p<\infty$ we denote the "energy" by
\begin{equation*}
	\mathcal{E}_{p,K}(\phi) = \iint_{Z^2} |\phi(\xi)-\phi(\omega)|^p K(\xi,\omega)\,\mathrm{d}\nu(\xi)\mathrm{d}\nu(\omega),
\end{equation*}
and finally the "domain" by
\begin{equation*}
	W^{p,K}(Z) = \left\{\phi\in L^p(Z)\mid \mathcal{E}_{p.K}(\phi)<\infty \right\}.
\end{equation*}
When appropriate, we endow $W^{p,K}(Z)$ with a norm $\|\phi|W^{p,K}\|^p=\|\phi|L^p\|^p+\mathcal{E}_{p.K}(\phi)$.

\begin{proposition}\label{prop:compact-embedding-general}
	Let $(Z,\nu)$ be a probability measure space, and let $K:Z^2\to[0,\infty)$ be locally integrable on $Z^2\setminus\{\text{diagonal}\}$. Let $1\leq p<\infty$.

	Suppose that there is a sequence $K_n:Z^2\to[0,\infty]$ of kernels, such that for all $n\in\mathbb{N}$:
	\begin{enumerate}
		\item $K_n\leq K$ (pointwise);
		\item\label{item:cpct-emb-Tn-cpct} $T_{K_n}$ is a compact operator on $L^p(\nu)$;
		\item\label{item:Nn-defn} denoting $N_n(\xi)=\int_Z K_n(\xi,\omega)\,\mathrm{d}\nu(\omega)$, we have $\ess\inf N_n \to\infty$ as $n\to\infty$;
	\end{enumerate}
	Then $W^{p,K}$ is compactly embedded in $L^p(\nu)$.
\end{proposition}

\begin{proof}
	For a.e.~$\xi\in Z$, we have that $K_n(\xi,\cdot)/N_n(\xi)\,\mathrm{d}\nu(\cdot)$ is a probability measure on $Z$; in the next calculation, we use H\"older's inequality for these measures. Let $\phi\in W^{p,K}(Z)\subseteq L^p(Z)$. Then
	\begin{align*}
		\big\|\phi-\tfrac{1}{N_n}T_{K_n}(\phi) \mid L^p\big\|^p &=
		\int_Z \Big|\phi(\xi)- \tfrac{1}{N_n(\xi)}\int_Z K_n(\xi,\omega)\phi(\omega)\,\mathrm{d}\nu(\omega)\Big|^p\mathrm{d}\nu(\xi)\\
		&= \int_Z \Big|\int_Z\left(\phi(\xi)- \phi(\omega)\right)\cdot\frac{K_n(\xi,\omega)}{N_n(\xi)}\,\mathrm{d}\nu(\omega)\Big|^p\mathrm{d}\nu(\xi)\\
		&\leq \int_Z\int_Z |\phi(\xi)-\phi(\omega)|^p \cdot \frac{K_n(\xi,\omega)}{N_n(\xi)}\,\mathrm{d}\nu(\omega)\mathrm{d}\nu(\xi)\\
		&\leq (\ess\inf N_n)^{-1}\cdot \iint_{Z^2} |\phi(\xi)-\phi(\omega)|^p K_n(\xi,\omega)\,\mathrm{d}\nu(\omega)\mathrm{d}\nu(\xi)\\
		&\leq (\ess\inf N_n)^{-1}\cdot \mathcal{E}_{p,K}(\phi).
	\end{align*}
	Since the constant in the last formula tends to $0$ as $n\to\infty$, and $\frac{1}{N_n}T_{K_n}$ is compact, it follows that we can approximate any bounded set in $W^{p,K}(Z)$ by a compact set in $L^p(Z)$, in $L^p$-norm. Hence the inclusion $W^{p,K}(Z) \to L^p(Z)$ is compact.
\end{proof}

\begin{proposition}\label{prop:compactness-of-kernel-ops}
	Let $(Z,d,\nu)$ satisfy \ref{ass:1}, i.e.~it is an Ahlfors--David $D$-regular bounded metric measure space. For $\eps>0$, denote by $K_\eps:Z^2\to[0,\infty)$ the function $K_\eps = \chi_{\{(\xi,\omega)\in Z^2 | d(\xi,\omega)>\eps\}}\cdot d^{-D}$. Then for any $1<p<\infty$, $T_{K_\eps}$ is compact on $L^p(Z)$. Furthermore, the function $N_\eps$ (see \ref{prop:compact-embedding-general}/\eqref{item:Nn-defn}) satisfies $\ess\inf N_\eps \to\infty$ as $\eps\to 0^+$.
\end{proposition}

\begin{proof}
	As each $K_\eps$ is a bounded function on a finite measure space $Z^2$, the integral kernel operators $T_{K_\eps}$ are $(L^p-L^p)$-Hille--Tamarkin operators \cite[p. 403]{Jacob-Evans:2020}, and hence compact \cite[Theorem 15.15]{Jacob-Evans:2020}. This shows the first claim.

	For the second claim, Lemma \ref{lem:sobolev-for-a-ball} gives $N_\eps \succ \log(\diam(Z)/\eps)\to\infty$ as $\eps\to 0^+$.
\end{proof}

Proposition \ref{prop:main-cpct-embedding} now immediately follows from Propositions \ref{prop:compact-embedding-general} and \ref{prop:compactness-of-kernel-ops}.

\section{Mixing}\label{sec:mixing}

The aim of this section is to prove Propositions from the Introduction related to mixing: Proposition \ref{prop:main-subrep-of-ub} about unboundedness, the $p>1$ case of Proposition \ref{prop:mixing-in-intro} about mixing, and Proposition \ref{prop:pi-p-on-L-q-expo-growth} below about exponential growth.

\begin{proof}[Proof of Proposition \ref{prop:main-subrep-of-ub}]
	Suppose, for contradiction, that $\pi$ is uniformly bounded on $W$. Let $w\in W\setminus\{0\}$, and denote $\theta_0,\theta_1,\dots$ a dense subset of $V^*$. Note that one can apply the weak mixing condition to finitely many functionals $\theta$ at the same time; and use this inductively to obtain a sequence $(g_n)_{n=0}^{\infty}\subset \Gamma$, such that for all $n\in \mathbb{N}$, $|\theta_i(\pi(g_n)w)|<\frac{1}{n}$ for all $0\leq i\leq n$.

	As $\pi$ is assumed to be uniformly bounded on $W$, the set $\{\pi(g_n)w|n\in\mathbb{N}\}$ is bounded in $W$, hence pre-compact in $V$. Thus there exists a subsequence, continued to be denoted $(g_n)_{n\in\mathbb{N}}$, and a vector $v\in V$, such that $\pi(g_n)w\to v$ in norm in $V$. As $w\not=0$ and $\pi$ is uniformly bounded on $V$, the sequence $\|\pi(g_n)w\|$ is bounded away from $0$, and hence $v\not=0$.
	On the other hand, we also have that $\theta_i(v)=\lim_{n\to\infty}\theta_i(\pi(g_n)w)=0$ for all $i\in \mathbb{N}$. It follows that $v=0$ by density, which is a contradiction.
\end{proof}

We record some prelimiary observations, perhaps well--known to experts. In the rest of this section, we will denote $|g|=\rho(go,o)$ for $g\in\Gamma$ when working under Assumption \ref{ass:3}. 

\begin{proposition}\label{prop:int-g'-to-a-power}
	Suppose \ref{ass:3}.
	Then we have the following estimates for $\Xi_T(g)=\int_Z \D{g}^T(\xi)\,\mathrm{d}\nu(\xi)$, with the implied constants depending on $T\geq0$ but not on $g\in\Gamma$:
	\begin{enumerate}
	\item For $T=0$ or $T=D$, $\Xi_T(g)=1$.
	\item\label{case:2-of-prop-int-g'} For $0<T<\frac{D}{2}$, $\Xi_T(g)\asymp \exp(-\eps T|g|)$.
	\item\label{case:3-of-prop-int-g'} For $T=\frac{D}{2}$, $\Xi_T(g) \asymp |g|\cdot\exp(-\eps T|g|)$.
	\item\label{case:4-of-prop-int-g'} For $T>\frac{D}{2}$, $\Xi_T(g) \asymp \exp(\eps(T-D)|g|)$.
	\end{enumerate}
\end{proposition} 

\begin{proof}
	First, the claim is trivial when $T=0$, and the change of variable formula gives the case $T=D$.

	Next, we calculate
	\begin{align*}
		\int_Z \D{g}^{T}(\xi)\,\mathrm{d}\nu(\xi)
		&=
		\int_Z \exp\left(\epsilon T(2\langle \xi,go\rangle_o-|g|)\right) \mathrm{d}\nu(\xi)\\
		&= \exp(-\epsilon T|g|)\cdot \int_Z \exp\left(2\epsilon T\langle \xi,go\rangle_o)\right) \mathrm{d}\nu(\xi).
	\end{align*}
	For this, we use a cover of $\partial X$ from \cite[Proof of Proposition 20]{Nica:2013} into `annuli about $go$' (alternatively, one could use the Shadow Lemma directly): there exists a spatial constant $\alpha>0$ and sets $A_k=\{\xi\in \partial X\mid k\alpha \leq \langle \xi,go\rangle_o < (k+1)\alpha \}$ for $0\leq k\leq N(g)$ (more about the integer $N(g)$ below) such that $\nu(A_k)\asymp \exp(-D\epsilon\alpha k)$. Thus we have
	\begin{equation}\label{localeq:int-g'-sum}
		\int_Z \exp\left(2\epsilon T\langle \xi,go\rangle_o)\right) \mathrm{d}\nu(\xi)
		\asymp
		\sum_{k=0}^{N(g)} \exp\left((2T-D)\epsilon \alpha k \right).
	\end{equation}
	When $2T-D<0$, i.e.~$T<\frac{D}{2}$, the geometric series is convergent, and case \eqref{case:2-of-prop-int-g'} follows.

	For the remaining cases, we need to recall from \cite[Proof of Proposition 20]{Nica:2013} that $N(g)=\lfloor (\sup_{\xi\in \partial X}\langle \xi, go\rangle_o - \delta)/\alpha \rfloor -1$ (where $\delta$ is the hyperbolicity constant). It follows that $N(g) +1 \leq |g|/\alpha$.

	When $T=\frac{D}{2}$, by the above facts about $N(g)$, the sum in \eqref{localeq:int-g'-sum} is $\asymp |g|$, and we are done with \eqref{case:3-of-prop-int-g'}.

	When $T>\frac{D}{2}$, the sum in \eqref{localeq:int-g'-sum} is $\asymp \exp((2T-D)\epsilon\alpha N(g))\leq \exp((2T-D)\epsilon|g|)$, and we have obtained the $\prec$ direction of \eqref{case:4-of-prop-int-g'}.
	To obtain also $\succ$, by Proposition \ref{prop:extend-pt-to-boundary} there exists a spatial constant $C\geq 0$ so that $N(g)\geq |g|/\alpha -C$. The conclusion now follows.
\end{proof}

\begin{proof}[Proof of Proposition \ref{prop:mixing-in-intro}, case $p>1$]
Denote by $q$ the conjugate exponent to $p$, i.e. $\frac{1}{p}+\frac{1}{q}=1$. For $\phi\in L^p$ and $\psi\in L^q$, we want to show that
\(\lim_{g\to\infty}\langle \pi^{(p)}_{it}(g)\phi,\psi\rangle = 0\).
By density, we can assume that $\|\phi|L^{\infty}\|, \|\psi|L^{\infty}\| < \infty$. Then
\begin{align*}
	\left|\langle \pi^{(p)}_{it}(g)\phi,\psi\rangle\right|
	&=
	\left|\int_Z \Db{g^{-1}}^{\frac{D}{p}-it}(\xi)\cdot \phi(g^{-1}\xi)\cdot\overline{\psi(\xi)}\,\mathrm{d}\nu(\xi)\right|\\
	&\leq
	\|\phi|L^{\infty}\|\cdot \|\psi|L^{\infty}\| \cdot \int_Z \Db{g^{-1}}^{\frac{D}{p}}(\xi)\,\mathrm{d}\nu(\xi).
\end{align*}
As $g\to\infty \iff g^{-1}\to\infty$, it follows from Proposition \ref{prop:int-g'-to-a-power} that the integral above converges to $0$ as $g\to\infty$.
\end{proof}

\begin{remark}\label{rem:pi-p-dual-pi-q}
	If $1<p<\infty$ and $q$ are conjugate exponents (i.e.~$1/p+1/q=1$), a straightforward calculation with the change of variable formula implies that for any $\phi,\psi\in L^\infty$ and $z\in\mathbb{C}$ we have
	\(
		\langle \pi^{(p)}_{z}(g)\phi,\psi\rangle
		=
		\langle \phi,\pi^{(q)}_{-\overline{z}}(g^{-1})\psi\rangle.
	\)
	Thus, in particular, if $t\in\mathbb{R}$ and $1<r,r'<\infty$ with $1/r + 1/r'=1$, then the adjoint of $\pi^{(p)}_{it}(g):L^r\to L^r$ is $\pi^{(q)}_{it}(g^{-1}):L^{r'}\to L^{r'}$. 
\end{remark}

\begin{proposition}\label{prop:pi-p-on-L-q-expo-growth}
	If $1<p,r<\infty$, $p\not=r$, then $\pi^{(p)}_{it}$ grows exponentially on $L^r$; more precisely $\|\pi^{(p)}_{it}\mid L^r\to L^r\| \asymp \exp\left(\left|\frac{1}{p}-\frac{1}{r}\right|\cdot|g|\right)$ for all $g\in\Gamma$.
\end{proposition}

\begin{proof}
	If $p>r$, then by Remark \ref{rem:pi-p-dual-pi-q}, the norm of $\pi^{(p)}_{it}(g)$ on $L^r$ equals the norm of $\pi^{(p')}_{it}(g^{-1})$ on $L^{r'}$, there $p'$ and $r'$ are conjugate to $p$ and $r$ respectively. We then have $p'<r'$ and $1/p'-1/r'=1/r-1/p=|1/p-1/r|$, so it suffices to prove the claim when $p<r$.

	We estimate
	\begin{align*}
		\left\| \pi^{(p)}_{it}(g) \mid L^r\to L^r \right\|^r
		&\geq
		\| \pi^{(p)}_{it}(g)\mathbf{1} | L^r \|^r
		=
		\int_Z \Db{g^{-1}}^{\frac{D}{p}\cdot r}(\xi)\,\mathrm{d}\nu(\xi)
		=
		\Xi_{Dr/p}(g^{-1})
	\end{align*}
	with the notation from Proposition \ref{prop:int-g'-to-a-power}. The lower bound now follows from \eqref{case:4-of-prop-int-g'} of that Proposition.

	For the upper bound, take an arbitrary $\phi\in L^r$ and estimate:
	\begin{align*}
		\left\| \pi^{(p)}_{it}(g) \phi \right\|^r
		&\leq
		\int_Z \Db{g^{-1}}^{\frac{D}{p}\cdot r}(\xi)\cdot|\phi|^r(g^{-1}\xi)\,\mathrm{d}\nu(\xi)\\
		&=
		\int_Z \Db{g^{-1}}^{\frac{D}{p}\cdot r}(g\omega)\cdot \D{g}^D(\omega)\cdot |\phi|^r(\omega)\,\mathrm{d}\nu(\omega)\\
		&=
		\int_Z \Db{g^{-1}}^{D(\frac{r}{p}-1)}(g\omega)\cdot |\phi|^r(\omega)\,\mathrm{d}\nu(\omega)\\
		&\leq
		\left\|\Db{g^{-1}}^{D(\frac{1}{p}-\frac{1}{r})}\mid L^{\infty}\right\|^r\cdot \left\|\phi|L^r\right\|^r.
	\end{align*}
	We also have $\D{g}(\xi) = \exp(\epsilon(2\langle \xi,go\rangle_o-|g|))\leq \exp(\epsilon |g|)$, and since $|g|=|g^{-1}|$, the conclusion follows.
\end{proof}

\section{The case of $p=1$}\label{sec:L1-case}

For this section we work under Assumption \ref{ass:3}. We focus on the case $p=1$ and $t=0$, i.e.~the representation $\pi^{(1)}_{0}$. It is not mixing on the whole of $L^1$, since $\langle \pi^{(1)}_0(g)\mathbf{1},\mathbf{1} \rangle = \|\pi^{(1)}_0(g)\mathbf{1}|L^1\| = \int_Z \D{g}^D(\xi)\,\mathrm{d}\nu(\xi) = 1$ for all $g\in\Gamma$.

Denote by $L^1_0 = \{ \phi\in L^1 \mid \int_Z \phi\,\mathrm{d}\nu = 0\}$ the codimension $1$ subspace of functions with zero average.
Denote $W^{\log,1}_0(Z) = W^{\log,1}(Z)\cap L^1_0$, a (closed and codimension $1$) subspace of $W^{\log,1}(Z)$. Observe that $\pi^{(1)}_{0}$ leaves $W^{\log,1}_0(Z)$ invariant (this uses just the change of variable formula \eqref{eq:change-of-var}, cf.~\eqref{eq:pi-p-preserves-Lp-norm}); denote the restriction as $\pi=\pi^{(1)}_{0}|_{W^{\log,1}_0}$. Finally, note that that our proof of the lower bound on the norm of $\pi^{(1)}_{0}$ in Theorem \ref{thm:main-s=0} uses $\pi^{(1)}_{0}(\cdot)\mathbf{1}$, hence does not apply to $\pi$.

\begin{question}
	Assuming \ref{ass:3}, is $\pi$ uniformly bounded?
\end{question}

Considering Proposition \ref{prop:main-subrep-of-ub}, it is currently not known to the authors whether $W^{\log,1}_0$ is compactly embedded in $L^1_0$, however $\pi^{(1)}_{0}$ is weakly mixing on $L^1_0$; this is the $p=1$ case of Proposition \ref{prop:mixing-in-intro}.

\begin{proposition}\label{prop:pi^1_0-is-mixing-on-L^1_0}
	The representation $\pi^{(1)}_{0}$ is weakly mixing on $L^1_0(\partial X)$.
\end{proposition}

We defer the proof until the end of the Section. Denoting $b:\Gamma\to W^{\log,1}(Z)$, $b(g)=\pi^{(1)}_0(g)(\mathbf{1})-\mathbf{1}$, we see that $b$ is in fact a cocycle for $\pi$, i.e. $b(g)\in W^{\log,1}_0(Z)$. Furthermore, the proof of the lower norm bound for $\pi^{(1)}_0$ in fact establishes that this cocycle is proper; $\|b(g)\|\asymp \rho(go,o)$ (cf.~Proposition \ref{prop:LlogL-of-pi-g-1-lower-bound}).

A positive answer to the above Question  would provide an alternative construction for a proper uniformly Lipschitz action of hyperbolic groups on a closed subspace of $L^1$; this has been established by other methods in \cite{Drutu-Mackay:2023,Vergara:2021,Vergara:2023}.

\begin{proposition}
	Suppose \ref{ass:1}.
	Then $W^{\log,p}(Z)$ is a closed subspace of an $L^p$-space. Consequently, $W^{\log,1}_0(Z)$ is a closed subspace of an $L^1$-space.
\end{proposition}
\begin{proof}
	Define a measure on $Z\times Z$ by $\mathrm{d}\mu(\xi,\eta) = d^{-D}(\xi,\eta)\mathrm{d}\nu(\xi)\mathrm{d}\nu(\eta)$. By standard arguments, this is a $\sigma$-finite Borel measure on $Z\times Z$. For a measurable $\phi:Z\to\mathbb{C}$, denote $D(\phi)(\xi,\eta) = \phi(\xi)-\phi(\eta)$.
	Then the linear map
	\begin{align*}
		T&:W^{\log,p}(Z)\to L^p(Z)\oplus_{p} L^p(Z\times Z; \mu), &\phi\mapsto (\phi,D(\phi))
	\end{align*}
	is an isometry (just from the definitions of the norms), and thus has a closed range.
\end{proof}

We close the section by proving Proposition \ref{prop:pi^1_0-is-mixing-on-L^1_0}, however we need to recall a result from \cite{Bader-Furman:2017}. Suppose \ref{ass:3}.
We say that a sequence $(x_{n})_{n\in\mathbb{N}}\subset X$ \emph{converges non-tangentially} to $\xi\in\partial X$, if there exists $M\geq 0$, such that $x_{n}\to\xi$ and $\langle x_{n},\xi\rangle_o\geq d(x_n,o)-M$. For a non-example, consider the free group $X=F_2=F(a,b)$, endowed the word metric $d$ from the standard generating set, $o=1$. Let $x_{n}=a^{2n}b^{n}$, $\xi=\lim_{n\to\infty} x_{n}=\lim_{n\to\infty} a^n$. Then $d(x_n,o)=3n$, but $\langle x_n,\xi\rangle_o=2n$.

In CAT($-1$) setting, points $\xi\in\partial X$ for which there exists a sequence in $X$ converging to it non-tangentially are often called \emph{conical limit points}, see e.g.~\cite[page 17]{Roblin:2003}. In our setting every $\xi\in\partial X$ admits such a sequence, see e.g.~\cite[pages 14-15]{Bader-Furman:2017} for a construction.

\begin{lemma}[{\cite[Lemma 5.1]{Bader-Furman:2017}}]\label{lem:Bader-Furman-SAT}
	Suppose \ref{ass:3}.
	For any $\phi\in L^1(\nu)$, there exists $Z_{\phi}\subseteq \partial X$ with $\nu(Z_{\phi})=1$, such that for any $(g_n)_{n\in\mathbb{N}}\subseteq \Gamma$ with $(g_n^{-1}o)_{n\in\mathbb{N}}$ converging non-tangentially to $\xi\in Z_{\phi}$, we have
	\begin{align*}
		\lim_{n\to\infty} \int_{\partial X} \left| \phi(g_n^{-1}\eta) - \phi(\xi) \right|\mathrm{d}\nu(\eta) =0.
	\end{align*}
\end{lemma}
We stated the Lemma slighly more generally, however the same proof as in \cite{Bader-Furman:2017} carries through. This is essentially the statement that the action of $\Gamma$ on $\partial X$ is \emph{strongly almost transitive}, see \cite[Section 2]{Bader-Furman:2017}.

\begin{corollary}\label{coro:weak-convergence-to-flat}
	Suppose \ref{ass:3}.
	For any $\phi\in L^{\infty}(\nu)$ there exists $Z_{\phi}\subset \partial X$ with full measure, such that for any $(g_{n})_{n\in\mathbb{N}}\subseteq \Gamma$ with $g_n^{-1}o\to \xi\in Z_{\phi}$ non-tangetially, and any $\psi\in L^{1}(\nu)$, we have
	\begin{equation}\label{eq:weak-convergence-to-flat}
		\langle \pi^{(1)}_{0}(g_n^{-1})\psi,\phi\rangle \to \phi(\xi)\cdot \textstyle\int_{\partial X}\psi\,\mathrm{d}\nu.
	\end{equation}
\end{corollary}

\begin{proof}
	As $L^{\infty}(\nu)$ is dense in $L^1(\nu)$ and $\pi^{(1)}_{it}$ is isometric on $L^1(\nu)$, we can assume that $\psi\in L^{\infty}$. Let $Z_{\phi}$ be as in Lemma \ref{lem:Bader-Furman-SAT}. Take a $(g_n)\subseteq\Gamma$ and $\xi\in Z_{\phi}$ as in the statement. Now
	\begin{align*}
		\langle \pi^{(1)}_{0}(g_n^{-1})\psi,\phi\rangle
		&=
		\int_{\partial X} \D{g_n}^{D}(\eta)\cdot \psi(g_n\eta)\cdot\phi(\eta)\,\mathrm{d}\nu(\eta)
		=\int_{\partial X} \psi(\omega)\cdot\phi(g_n^{-1}\omega)\,\mathrm{d}\nu(\omega),
	\end{align*}
	and thus
	\begin{align*}
		\left|\langle \pi^{(1)}_{0}(g_n^{-1})\psi,\phi\rangle - \phi(\xi)\cdot \textstyle\int_{\partial X}\psi\,\mathrm{d}\nu\right|
		&\leq
		\int_{\partial X} |\psi(\omega)|\cdot \left| \phi(g_n^{-1}\omega) -\phi(\xi) \right| \mathrm{d}\nu(\omega)\\
		&\leq \|\psi|L^{\infty}\|\cdot \int_{\partial X} \left| \phi(g_n^{-1}\omega) -\phi(\xi) \right| \mathrm{d}\nu(\omega),
	\end{align*}
	which converges to $0$ by Lemma \ref{lem:Bader-Furman-SAT}.
\end{proof}

Proposition \ref{prop:pi^1_0-is-mixing-on-L^1_0} now immediately follows from the above Corollary: for any $\psi\in L^1_0$ and $\phi\in L^{\infty}$, for any sequence $(g_n)_{n\in\mathbb{N}}$ as in the above Corollary, the matrix coefficients will converge to the right-hand side of \eqref{eq:weak-convergence-to-flat}, which is zero.

\bibliographystyle{plain}
\bibliography{log-sobolev}

\end{document}